\documentclass[12pt,a4paper]{amsart}
\usepackage{amssymb,amsmath}
\usepackage{hyperref}

\numberwithin{equation}{section}
\newtheorem{theorem}{Theorem}[section]
\newtheorem{lemma}[theorem]{Lemma}
\newtheorem{proposition}[theorem]{Proposition}
\newtheorem{corollary}[theorem]{Corollary}

\theoremstyle{definition}
\newtheorem{definition}[theorem]{Definition}
\newtheorem{example}[theorem]{Example}
\theoremstyle{remark}
\newtheorem{remark}[theorem]{Remark}

\newcommand\Apr{\operatorname{Apr}}

\begin{document}
\title[Roughness in quotient groups]{Roughness in quotient groups}%
\author[W. Mahmood]{Waqas Mahmood }
\address{Quaid-I-Azam University Islamabad, Pakistan}%
\email{wmahmood$@$qau.edu.pk}

\thanks{This research was partially supported by Higher Education Commission, Pakistan}

\maketitle
\begin{abstract}
The theory of rough sets was firstly introduced by Pawlak (see \cite{p}). Many Mathematician has been studied the relations between rough sets and algebraic systems such as groups, rings and modules. In this paper we will introduce the lower and upper approximations in a quotient group. We will discuss several properties of the lower and upper approximations. Moreover under some additional assumptions we are able to show that the lower approximation is a normal subgroup of the quotient group but this property fails for the upper approximation. At the end we will develop several homomorphisms between  lower approximations.
\end{abstract}

\section{Introduction}
A rough set is a subset of a universe which is defined by a pair of ordinary sets called lower and upper approximation. The theory of rough sets is an extension of the set theory, for dealing with the ambiguity in information systems. Combining the theory of rough sets with abstract algebra is a way generalizing it. In recent years, most of the sets are based on the imprecise information. To analyze any such kind of information, mathematical logics are most helpful. On the other hand some authors has studied the rough algebraic structure. In \cite{b}, \cite{d} and \cite{m} the concept of rough groups, rough quotient groups are studied. In \cite{k} N. Kuroki and P. Wang introduced the notion of the lower and upper approximations with respect to a normal subgroup in a group. They have proved several properties of them. Moreover they also defined the lower and upper approximations with respect to a $t$-level subset of a fuzzy normal subgroup.

In \cite{k1} N. Kuroki introduced the rough left (resp. right and bi-) ideals in a semigroup. He has also defined the lower and upper approximations of a quotient semigroup with respect to a congruence relation over a semigroup. N. Kuroki proved that these are left (resp. right and bi-) ideals in the quotient semigroup. Q. M. Xiao and Z. Zhang have discussed the relations between the upper (lower) rough prime ideals and rough fuzzy prime ideals in a semigroup (see \cite{q}).

After that B. Davvaz introduced the concept of rough rings and ideals (see \cite{da}). He introduced the notion of rough sub-ring (resp. ideal) with respect
to an ideal of a ring which is an extended notion of a sub-ring (resp. ideal) in a ring. Also he has shown several properties of the lower and upper approximations with respect to an ideal in a ring.

In \cite{da2} B. Davvaz has defined the upper rough ideal in a ring $R$ with respect to a $t$-level congruence relation of a fuzzy idela on R. In \cite{b1} XU Bi-cai introduced anti-homomorphism of a group. Also P. Isaac and Neelima has studied some properties about rough ring homomorphism and anti-homomorphism (see \cite{i}). Moreover in \cite{h} S. Han, W. Cheng and J. Wang defined the rough ring in an approximation space. Rough modules are introduced and defined by Davvaz and Mahdavipour in 2006. They have studied some properties of the lower and the upper approximations in a rough module.

In the recent paper we shall introduce the notion of the lower and upper approximations in a quotient group. Then we will prove serval properties of them such as intersection, union and product. Moreover it is shown that the lower and upper approximations of a normal subgroup of a quotient group does not provide us any new information. After that we will produce some homomorphisms between the lower approximation spaces.
\section{Lower and upper approximations in a quotient groups}\label{q}
First of all in this section we will recall the notation of the rough sets.

\begin{definition}
Let $\varnothing\neq U$ be a universe and $\theta$ an equivalence relation over $U$. Then the pari $(U,\theta)$ is called an approximation space.
\end{definition}

\begin{definition}
If $(U,\theta)$ is an approximation space then the mapping $\Apr: P(U)\to P(U)\times P(U)$ defined by
\[
\Apr(X)=(\underline{X},\overline{X}), \text{ for all } X\in P(U)
\]
is called rough approximation operator. Here $\underline{X}:\{x\in U: [x]_\theta \subseteq X\}$ and $\overline{X}:\{x\in U: [x]_\theta \cap X\neq \varnothing\}$ are called lower and upper rough approximations of $X$ in $(U,\theta)$ respectively.
\end{definition}

Note that it is clear from the definition that $\underline{X}\subseteq X\subseteq \overline{X}.$

\begin{definition}
For a given approximation space $(U,\theta)$, a pair $(A,B)\in P(U)\times P(U)$ is called a rough set if'f $(A,B)=\Apr(X)$ for some $X\in P(U)$.
\end{definition}

Throughout this paper $G$ will be denoted as a group under multiplication with identity element $e$. Let $N$ be any normal subgroup of $G$. Now we define a relation $\theta$ over $G/N$ as follows:
\[
xN{\theta} yN \Leftrightarrow xN=(aN)(yN)(aN)^{-1}=(aya^{-1})N \text{ for some } a\in G.
\]
Note that $xN{\theta} yN$ if and only if they are conjugates in $G/N$. It is well known that $\theta$ is an equivalence relation over the quotient group $G/N$. Hence $(G/N,\theta)$ is an approximation space. We will denote the equivalence class of $xN\in G/N$ by $[xN]_{\theta}$. So $[xN]_{\theta}$ is the set of all conjugates of $xN$ in $G/N$. That is we have:
\[
[xN]_{\theta}=\{(aN)(xN)(aN)^{-1}:aN\in G/N\}=\{(axa^{-1})N:a\in G\}.
\]
First of all we will prove the following Lemma:

\begin{lemma}\label{1}
With the previous notation we have:
\[
[x_1x_2N]_{{\theta}}=[(x_1N)(x_2N)]_{{\theta}}\subseteq [x_1N]_{{\theta}}[x_2N]_{{\theta}}
\]
for all $x_1N,x_2N\in G/N$.
\end{lemma}

\begin{proof}
Since $N$ is normal so $x_1x_2N=(x_1N)(x_2N)$. It proves the first equality. Now let $gN\in [x_1x_2N]_{{\theta}}$ then $gN=(ax_1x_2a^{-1})N$ for some $a\in G$. By normality of $N$ we get that
\[
gN=(ax_1x_2a^{-1})N=(ax_1a^{-1})(ax_2a^{-1})N=((ax_1a^{-1})N)((ax_2a^{-1})N))
\]
By definition $(ax_ia^{-1})N\in [x_iN]_{{\theta}}$ for all $i=1,2$. It follows that $gN\in [x_1N]_{{\theta}}[x_2N]_{{\theta}}$. It provides the following inclusion:
\[
[x_1x_2N]_{{\theta}}\subseteq [x_1N]_{{\theta}}[x_2N]_{{\theta}}.
\]
\end{proof}

Let $H$ be any subset of $G$ and $N$ a normal subgroup of $G$ such that $N\subseteq H$. Then $H/N\subseteq G/N$ is the set of all those elements $aN\in G/N$ such that $a\in H$. Moreover if $K\subseteq G$ with $N\subseteq K$ then $(H/N)(K/N)$ is the following set:
\[
(H/N)(K/N)=\{(hN)(kN):h\in H,k\in K\}=\{hkN:h\in H,k\in K\}
\]
It is clear that $(H/N)(K/N)=(HK)/N$. Now we define the following sets of $G/N$:
\[
\underline{\Apr}_{G/N}(H/N):=\{xN\in G/N:[xN]_{{\theta}}\subseteq H/N\}, \text{ and }
\]
\[
\overline{\Apr}_{G/N}(H/N):=\{xN\in G/N:[xN]_{{\theta}}\cap H/N\neq \varnothing\}.
\]
Then $\underline{\Apr}_{G/N}(H/N)$ (resp. $\overline{\Apr}_{G/N}(H/N)$) is called the lower (resp. upper) approximation of $H/N$ with respect to $N$ in the approximation space $(G/N,\theta)$.

\begin{lemma}\label{10}
With the above notation suppose that $H_i$ is a non-empty subset of $G$ such that $N\subseteq H_i$ for all $i=1,2$. If $H_1/N\subseteq H_2/N$ then
\[
\underline{\Apr}_{G/N}(H_1/N)\subseteq \underline{\Apr}_{G/N}(H_2/N)\text{ and }
\]
\[
\overline{\Apr}_{G/N}(H_1/N)\subseteq \overline{\Apr}_{G/N}(H_2/N).
\]
\end{lemma}

\begin{proof}
let $xN\in \overline{\Apr}_{G/N}(H_1/N)$ then $[xN]_{{\theta}}$ and $H_1/N$ have a non-empty intersection. Since $H_1/N\subseteq H_2/N$ so it follows that
\[
[xN]_{{\theta}}\cap H_2/N\neq \varnothing.
\]
By definition of the upper approximation we conclude that $xN\in \overline{\Apr}_{G/N}(H_2/N)$. Hence $\overline{\Apr}_{G/N}(H_1/N)$ is a subset of $\overline{\Apr}_{G/N}(H_2/N)$. By the similar arguments we can prove the other inclusion.
\end{proof}

\begin{proposition}\label{2}
Suppose that $N$ is a normal subgroup of $G$ and $H_i\subseteq G$ such that $N\subseteq H_i$ for all $i=1,2$ . Then the following conditions hold:
\begin{itemize}
\item [(1)] $\varnothing\neq \underline{\Apr}_{G/N}(H_1/N)\subseteq H_1/N\subseteq \overline{\Apr}_{G/N}(H_1/N).$
\item [(2)] $\overline{\Apr}_{G/N}((H_1\cup H_2)/N)=\overline{\Apr}_{G/N}(H_1/N)\cup \overline{\Apr}_{G/N}(H_2/N).$
\item [(3)] $\underline{\Apr}_{G/N}(H_1/N)\cup \underline{\Apr}_{G/N}(H_2/N)\subseteq \underline{\Apr}_{G/N}((H_1\cup H_2)/N)$.
\item [(4)] $\overline{\Apr}_{G/N}((H_1\cup H_2)/N)=\overline{\Apr}_{G/N}((H_1/N)\cup (H_2/N))$.
\item [(5)] $\underline{\Apr}_{G/N}((H_1\cup H_2)/N)=\underline{\Apr}_{G/N}((H_1/N)\cup (H_2/N))$.
\item [(6)] If $H_1\subseteq H_2$ then $\underline{\Apr}_{G/N}(H_1/N)\subseteq \underline{\Apr}_{G/N}(H_2/N)$.
\item [(7)] If $H_1\subseteq H_2$ then $\overline{\Apr}_{G/N}(H_1/N)\subseteq \overline{\Apr}_{G/N}(H_2/N)$.
\end{itemize}

\end{proposition}

\begin{proof}

$(1)$ It is obvious. Note that $e\in N\subseteq H_1$ and $[N]_{{\theta}}=\{N\}\subseteq H_1/N.$

$(2)$ Since $N\subseteq H_i$ for all $i=1,2$ so it follows that $N\subseteq H_1\cup H_2$. Then $xN\in \overline{\Apr}_{G/N}(H_1/N)\cup \overline{\Apr}_{G/N}(H_2/N)$
\[
\Leftrightarrow xN\in \overline{\Apr}_{G/N}(H_1/N) \text{ or } xN\in \overline{\Apr}_{G/N}(H_2/N).
\]
\[
\Leftrightarrow [xN]_{{\theta}}\cap H_1/N\neq \varnothing \text{ or } [xN]_{{\theta}}\cap H_2/N\neq \varnothing.
\]
\[
\Leftrightarrow  \text{ There exists } yN\in G/N  \text{ such that } yN\in H_1/N  \text{ or } yN\in H_2/N  \text{ and } xN{\theta} yN.
\]
\[
\Leftrightarrow yN=hN  \text{ for some } h\in H_1  \text{ or } h\in H_2 \text{ and } xN{\theta} yN.
\]
\[
\Leftrightarrow yN=hN  \text{ for some } h\in H_1\cup H_2 \text{ and } xN{\theta} yN.
\]
\[
\Leftrightarrow yN\in (H_1\cup H_2)/N \text{ and } xN{\theta} yN.
\]
\[
\Leftrightarrow [xN]_{{\theta}}\cap (H_1\cup H_2)/N\neq \varnothing.
\]
\[
\Leftrightarrow xN\in \overline{\Apr}_{G/N}((H_1\cup H_2)/N).
\]
This proves the equality.

$(3)$ Since $H_i\subseteq H_1\cup H_2$ for all $i=1,2$ it implies that $H_i/N\subseteq (H_1\cup H_2)/N$ for all $i=1,2$. Then by Lemma \ref{10} we have $\underline{\Apr}_{G/N}(H_i/N)\subseteq \underline{\Apr}_{G/N}((H_1\cup H_2)/N)$ for all $i=1,2$. Hence
\[
\underline{\Apr}_{G/N}(H_1/N)\cup \underline{\Apr}_{G/N}(H_1/N)\subseteq \underline{\Apr}_{G/N}((H_1\cup H_2)/N).
\]
Note that $(4)$ and $(5)$ are obvious in view of the fact that $(H_1\cup H_2)/N=(H_1/N)\cup (H_2/N).$

$(6)$ Since $H_1\subseteq H_2$ it implies that $H_1/N\subseteq H_2/N$. Then by Lemma \ref{10} we have
\[
\underline{\Apr}_{G/N}(H_1/N)\subseteq \underline{\Apr}_{G/N}(H_2/N).
\]

$(7)$ It can be proved by the same arguments as we used in $(6)$.
\end{proof}

In the following we will show that $\underline{\Apr}_{G/N}(H_1/N)$ and $\overline{\Apr}_{G/N}(H_1/N)$ are not subgroups of $G/N$ in general.

\begin{example}\label{22}
Let $G=Q_8=\{\pm 1,\pm i,\pm j,\pm k\}$, the quaternion group with the following relations:
\[
i^2=j^2=k^2=-1, ij=k, jk=i \text{ and } ki=j.
\]
Let $N=\{\pm 1\}$ and $H=N\cup \{i,j\}$. Then it is clear that $N$ is a normal subgroup of $G$ such that $N\subseteq H$. Moreover $H$ is not a subgroup of $G$. It can be shown that $iN=-iN$, $jN=-jN$ and $kN=-kN$. Then we have
\[
G/N=\{N,iN,jN,kN\}\text { and } H/N=\{N,iN,jN\}.
\]
Now $[aN]_{{\theta}}=\{(xax^{-1})N:x\in G\}.$ It follows that
\[
[iN]_{{\theta}}=\{iN\},[jN]_{{\theta}}=\{jN\},[kN]_{{\theta}}=\{kN\}.
\]
Hence $\underline{\Apr}_{G/N}(H/N)=H/N=\overline{\Apr}_{G/N}(H/N).$ Therefore $\underline{\Apr}_{G/N}(H_1/N)$ and $\overline{\Apr}_{G/N}(H_1/N)$ are not subgroups of $G/N$.
\end{example}

\begin{proposition}\label{5}
Suppose that $H_i\subseteq G$ with $N\subseteq H_i$ for all $i=1,2$ where $N$ is a normal subgroup of a group $G$. Then we have:
\begin{itemize}
\item [(1)] $\overline{\Apr}_{G/N}((H_1\cap H_2)/N)\subseteq \overline{\Apr}_{G/N}((H_1/N)\cap (H_2/N))\subseteq \overline{\Apr}_{G/N}(H_1/N)\cap \overline{\Apr}_{G/N}(H_2/N)$.
\item [(2)] $\underline{\Apr}_{G/N}((H_1\cap H_2)/N)\subseteq \underline{\Apr}_{G/N}((H_1/N)\cap (H_2/N))= \underline{\Apr}_{G/N}(H_1/N)\cap \underline{\Apr}_{G/N}(H_2/N)$.
\end{itemize}
\end{proposition}

\begin{proof}
Note that $(H_1\cap H_2)/N\subseteq (H_1/N)\cap (H_2/N)\subseteq H_i/N$ for all $i=1,2$. By Lemma \ref{10} it follows that
\[
\overline{\Apr}_{G/N}((H_1\cap H_2)/N)\subseteq \overline{\Apr}_{G/N}((H_1/N)\cap (H_2/N))\subseteq \overline{\Apr}_{G/N}(H_i/N) \text{ for all $i=1,2$.}
\]
Then we have
\[
\overline{\Apr}_{G/N}((H_1\cap H_2)/N)\subseteq \overline{\Apr}_{G/N}((H_1/N)\cap (H_2/N))\subseteq \overline{\Apr}_{G/N}(H_1/N)\cap \overline{\Apr}_{G/N}(H_2/N).
\]
By the above Similar arguments we can prove that
\[
\underline{\Apr}_{G/N}((H_1\cap H_2)/N)\subseteq \underline{\Apr}_{G/N}((H_1/N)\cap (H_2/N))\subseteq \underline{\Apr}_{G/N}(H_1/N)\cap \underline{\Apr}_{G/N}(H_2/N).
\]
Now let $xN$ be an arbitrary element of $\underline{\Apr}_{G/N}(H_1/N)\cap \underline{\Apr}_{G/N}(H_2/N)$ then $xN\in \underline{\Apr}_{G/N}(H_i/N)$ for all $i=1,2$. It implies that $[xN]_{{\theta}}\subseteq H_i/N$ for all $i=1,2$. Then $[xN]_{{\theta}}\subseteq (H_1/N)\cap (H_2/N)$. It follows that $xN\in \underline{\Apr}_{G/N}((H_1/N)\cap (H_2/N))$. This finishes the proof of the Proposition.
\end{proof}

\begin{corollary}\label{12}
With the notation of Proposition \ref{5} assume in addition that $H_1$ and $H_2$ are subgroups of $G$. Then the following are true:
\[
\overline{\Apr}_{G/N}((H_1\cap H_2)/N)= \overline{\Apr}_{G/N}((H_1/N)\cap (H_2/N))\text{ and }
\]
\[
\underline{\Apr}_{G/N}((H_1\cap H_2)/N)= \underline{\Apr}_{G/N}((H_1/N)\cap (H_2/N)).
\]
\end{corollary}

\begin{proof}
Suppose that $H_1$ and $H_2$ are subgroups of $G$. We claim that $(H_1/N)\cap (H_2/N)=(H_1\cap H_2)/N$. Note that $(H_1\cap H_2)/N\subseteq (H_1/N)\cap (H_2/N)$. Let $xN\in (H_1/N)\cap (H_2/N)$ then
\[
xN=h_1N=h_2N \text{ where } h_i\in H_i \text{ for all i=1,2.}
\]
It implies that $h_1h_2^{-1}\in N\subseteq H_i$ for all $i=1,2.$ Since $H_1$ and $H_2$ are subgroups so $h_1,h_2\in H_i$ for all $i=1,2.$ It follows that $xN\in (H_1\cap H_2)/N$. This proves the claim. Moreover it also provides us the required equalities.
\end{proof}

Note that in Corollary \ref{12} the condition of $H_1$ and $H_2$ are subgroups of $G$ is necessary. In the next we will give an example such that the results in Corollary are not true if we skip this condition.

\begin{example}
Let $G=A_4$, the alternating group and $N=\{I,(12)(34),(13)(24),(14)(23)\}$. Then $N$ is a normal subgroup of $G$. Let $H_1=N\cup \{(123),(124)\}$ and $H_2=N\cup \{(132),(142)\}$. Note that non of the $H_i$'s is a subgroup of $G$. Moreover $(132)N=(124)N$ and $(123)N=(142)N$. Then we have
\[
G/N=\{N,(123)N,(132)N\}=\{N,(142)N,(124)N\}\text{ and}
\]
\[
H_1/N=H_2/N=G/N.
\]
Now $H_1\cap H_2= N$ then $(H_1\cap H_2)/N=N/N\neq G/N=(H_1/N)\cap (H_2/N).$ Since $N$ and $G$ both are normal subgroups of $G$ so by the next Lemma \ref{6} it follow that
\[
\underline{\Apr}_{G/N}((H_1\cap H_2)/N)=N/N=\overline{\Apr}_{G/N}((H_1\cap H_2)/N), \text{ and }
\]
\[
\underline{\Apr}_{G/N}((H_1/N)\cap (H_2/N))=G/N=\overline{\Apr}_{G/N}((H_1/N)\cap (H_2/N)).
\]
This proves that the equality does note hold in the statement of the last Corollary \ref{12}.
\end{example}

In the following we will show that if $N$ and $H$ both are normal subgroups of $G$ such that $N\subseteq H$ then the lower and upper approximations of $H/N$ does not provide us any new information.

\begin{lemma}\label{6}
Suppose that $N$ and $H$ are normal subgroups of $G$ with $N\subseteq H.$ Then the following result hold:
\[
\underline{\Apr}_{G/N}(H/N)=H/N=\overline{\Apr}_{G/N}(H/N).
\]
\end{lemma}

\begin{proof}
To prove the result let $xN\in \overline{\Apr}_{G/N}(H/N)$ then there exists $yN\in G$ such that $yN=hN\in H/N$ for some $h\in H$ and $xN {\theta} yN$. It implies that $[xN]_{{\theta}}=[yN]_{{\theta}}=[hN]_{{\theta}}.$ Since $H$ is normal so it follows that $aha^{-1}N\in H/N$ for all $a\in G$ and $[hN]_{{\theta}}\subseteq H/N$. So we get that $xN\in \underline{\Apr}_{G/N}(H/N).$ So $\overline{\Apr}_{G/N}(H/N)\subseteq \underline{\Apr}_{G/N}(H/N)$. This proves the result in view of Proposition \ref{2}$(1)$.
\end{proof}

\begin{theorem}\label{3}
Let $G$ be a group and $N$ a normal subgroup of $G$. Then for any non-empty subsets $H_1$ and $H_2$ of $G$ containing $N$ the following statements hold:
\begin{itemize}
\item [(1)] $\overline{\Apr}_{G/N}((H_1H_2)/N)\subseteq \overline{\Apr}_{G/N}(H_1/N)\overline{\Apr}_{{G/N}}(H_2/N).$
\item [(2)] $\underline{\Apr}_{G/N} (H_1/N)\underline{\Apr}_{G/N} (H_2/N)\subseteq \underline{\Apr}_{G/N}((H_1H_2)/N).$
\end{itemize}
\end{theorem}

\begin{proof}
For the proof of $(1)$ let $xN\in \overline{\Apr}_{G/N}((H_1H_2)/N)$ then $[xN]_{{\theta}}\cap (H_1H_2)/N\neq \varnothing$. Then there exists $yN\in G/N$ such that $yN\in (H_1H_2)/N$ and $xN {\theta} yN$. It implies that $yN=(h_1h_2)N$ and $xN=(aya^{-1})N$ where $a\in G$ and $h_i\in H_i$ for all $i=1,2.$ Since $N$ is normal so we conclude that
\[
xN=(aya^{-1})N=(aN)(yN)(a^{-1}N)=(aN)((h_1h_2)N)(a^{-1}N)
\]
\[
=ah_1h_2a^{-1}N=(ah_1a^{-1})(ah_2a^{-1})N=(ah_1a^{-1}N)(ah_2a^{-1}N).
\]
Since $h_iN\in H_i/N$ for all $i=1,2$ so it induces the following fact:
\[
[ah_ia^{-1}N]_{{\theta}}\cap H_i/N= [h_iN]_{{\theta}}\cap H_i/N\neq \varnothing \text{ for all $i=1,2$.}
\]
Then by definition $ah_ia^{-1}N\in \overline{\Apr}_{G/N}(H_i/N)$ for all $i=1,2$. So we get that $xN$ belongs to the set $ \overline{\Apr}_{G/N}(H_1/N)\overline{\Apr}_{{G/N}}(H_2/N)$. It proves that $\overline{\Apr}_{G/N}((H_1H_2)/N)$ is a subset of $ \overline{\Apr}_{G/N}(H_1/N)\overline{\Apr}_{{G/N}}(H_2/N).$ This finishes the proof of $(1)$.

$(2)$ Suppose that $xN\in \underline{\Apr}_N (H_1/N)\underline{\Apr}_{G/N} (H_2/N)$. Then it implies that
\[
xN=(y_1N)(y_2N)=y_1y_2N\text{ where } y_iN\in \underline{\Apr}_{G/N}(H_i/N) \text{ for all $i=1,2$.}
\]
Then $[y_iN]_{{\theta}}\subseteq H_i/N$ for all $i=1,2$. By Lemma \ref{1} it follows that
\[
[xN]_{{\theta}}=[y_1y_2N]_{{\theta}}\subseteq [y_1N]_{{\theta}}[y_2N]_{{\theta}}.
\]
Now $[y_1N]_{{\theta}}[y_2N]_{{\theta}}\subseteq (H_1/N)(H_2/N)=(H_1H_2)/N$. It implies that $xN\in \underline{\Apr}_{G/N}((H_1H_2)/N)$. Hence this proves the claim in $(2)$.
\end{proof}

In the next example we will show that $\underline{\Apr}_{G/N} (H_1/N)\underline{\Apr}_{G/N} (H_2/N)$ is not equal to $\underline{\Apr}_{G/N}((H_1H_2)/N).$

\begin{example}\label{13}
Let $G=S_4$, the permutation group over the set $\{1,2,3,4\}$. Let $N=\{I,(12)(34),(13)(24),(14)(23)\}$, $H_1=N\cup \{(1324)\}$ and $H_2=N\cup \{(3412),(1243)\}$. Then $N$ is a normal subgroup such that $N\subseteq H_i$ for all $i=1,2$. It can be shown that:
\[
(3412)N=(2143)N=\{(13),(24),(2143),(3412)\},
\]
\[
(1324)N=(4231)N=\{(12),(34),(4231),(1324)\},
\]
\[
(1243)N=(3421)N=\{(14),(23),(3421),(1243)\},
\]
\[
(123)N=\{(123),(134),(124),(243)\}, \text{ and }
\]
\[
(132)N=\{(132),(143),(142),(234)\}.
\]
Since $|G/N|=6$ so it follows that
\[
H_1/N=\{N,(1324)N\},H_2/N=\{N,(3412)N,(1243)N\} \text{ and }
\]
\[
G/N=\{N,(1234)N,(1324)N,(1243)N,(123)N,(132)N\}.
\]
Moreover if $\alpha\in G$ is any permutation then $[\alpha N]_{{\theta}}=\{\beta N:\alpha \text{ and }\beta \text{ have same cycle structure}\}.$ So we have
\[
[(1234)N]_{{\theta}}=\{(1234)N,(1324)N,(1243)N\} \text{ and } [(123)N]_{{\theta}}=\{(132)N,(123)N\}
\]
Then we conclude that $\underline{\Apr}_{G/N}(H_i/N)=\{N\}=N/N$ for all $i=1,2$. Also note that $(H_1H_2)/N=\{N,(1324)N,(3412)N,(1243)N\}$. It implies that $\underline{\Apr}_{G/N}((H_1H_2)/N)=\{N,(1324)N,(3412)N,(1243)N\}$. Hence it follows that
\[
\underline{\Apr}_{G/N}(H_1/N)\underline{\Apr}_{G/N}(H_2/N)=N/N\neq \underline{\Apr}_{G/N}((H_1H_2)/N).
\]
\end{example}

\begin{definition}
Let $N\subseteq G$ be a normal subgroup and $H\subseteq G$ containing $N$. Suppose that $\Apr_{G/N}(H/N):=(\underline{\Apr}_{G/N}(H/N),\overline{\Apr}_{G/N}(H/N))$ and $\underline{\Apr}_{G/N}(H/N)\neq\overline{\Apr}_{G/N}(H/N)$ then $\Apr_{G/N}(H/N)$ is a rough set in the approximation space $(G/N,\theta)$.
\begin{itemize}
  \item [(1)] $\Apr_{G/N}(H/N)$ is called an upper rough subgroup (resp. normal subgroup) of $G/N$ if $\overline{\Apr}_{G/N}(H/N)$ is a subgroup (resp. normal subgroup) of $G/N$.
  \item [(2)] $\Apr_{G/N}(H/N)$ is called a lower rough subgroup (resp. normal subgroup) of $G/N$ if $\underline{\Apr}_{G/N}(H/N)$ is a subgroup (resp. normal subgroup) of $G/N$.
  \item [(3)] $\Apr_{G/N}(H/N)$ is called a rough subgroup (resp. normal subgroup) of $G/N$ if it is both upper and lower rough subgroup (resp. normal subgroup).
\end{itemize}
\end{definition}

As we have shown in Example \ref{22} that the lower and upper approximations are not subgroups. But in the next result we will prove that the lower approximation of $H/N$ is a normal subgroup of $G/N$ provided that $H$ is a subgroup of $G$.

\begin{proposition}\label{4}
Let $N$ be a normal subgroup and $H$ a subgroup of $G$ with $N\subseteq H$. Then $\Apr_{G/N}(H/N)$ is a lower rough normal subgroup of $G/N$.
\end{proposition}

\begin{proof}
By Proposition \ref{2}$(1)$ $\underline{\Apr}_{G/N}(H/N)\neq \varnothing$. Since $H/N$ is a subgroup so $[N]_{{\theta}}=\{N\}\subseteq H/N$ so $N\in \underline{\Apr}_{G/N}(H/N).$ It shows that identity element exists in $N\in \underline{\Apr}_{G/N}(H/N).$

Now let $x_iN\in \underline{\Apr}_{G/N}(H/N)$ for all $i=1,2$ then by definition we have $[x_iN]_{{\theta}}\subseteq H/N$ for all $i=1,2$. By Lemma \ref{1} it follows that
\[
[x_1x_2N]_{{\theta}}\subseteq [x_1N]_{{\theta}}[x_2N]_{{\theta}}\subseteq H/N.
\]
Recall that $H/N$ is a subgroup. So we have $(x_1N)(x_2N)=x_1x_2N\in \underline{\Apr}_{G/N}(H/N).$ This proves that closure law holds in $\underline{\Apr}_{G/N}(H/N).$

Now let $(g_1x_1^{-1}g_1^{-1})N\in [x_1^{-1}N]_{{\theta}}$ be an arbitrary element where $g_1\in G$. Since $N$ is normal so we have
\[
(g_1x_1^{-1}g_1^{-1})N=((g_1x_1g_1^{-1})N)^{-1}.
\]
Since $H/N$ is a subgroup and $(g_1x_1g_1^{-1})N\in H/N$ so it follows that $(g_1x_1^{-1}g_1^{-1})N\in H/N$. This proves that $[x_1^{-1}N]_{{\theta}}\subseteq H/N$. Then $x_1^{-1}N\in \underline{\Apr}_{G/N}(H/N)$. Moreover associative law holds in $G/N$ so it is also hold in $\underline{\Apr}_{G/N}(H/N)$. Hence $\underline{\Apr}_{G/N}(H/N)$ is a subgroup of $G/N.$

To prove normality let $gN\in G/N$ and $xN\in \underline{\Apr}_{G/N}(H/N)$ be any elements. Then it follows that
\[
gxg^{-1}N\in [gxg^{-1}N]_{{\theta}}=[xN]_{{\theta}}\subseteq H/N.
\]
Since $N$ is normal so by definition it implies that $(gN)(xN)(g^{-1}N)=gxg^{-1}N\in \underline{\Apr}_{G/N}(H/N)$. Consequently, this proves the normality of $\underline{\Apr}_{G/N}(H/N)$. Hence $\Apr_{G/N}(H/N)$ is a lower rough normal sub-group of $G/N$.
\end{proof}

The next example shows that $\Apr_{G/N}(H/N)$ is not an upper rough sub-group of $G/N$ even $H$ is a subgroup of $G$. Moreover it is also shown that the lower approximation of a subgroup $H/N$ is a proper normal subgroup of $H/N$.

\begin{example}\label{22}
Let $G=S_4$, $N=\{I,(12)(34),(13)(24),(14)(23)\}$ and $H=N\cup \{(12),(34),(1324),(4231)\}$. Then $N$ is a normal subgroup and $H$ is a subgroup of $G$ such that $N\subseteq H$. Moreover $H$ is not normal in $G$. Since $|H/N|=2$ and $|G/N|=6$ so by Example \ref{13} it follows that
\[
G/N=\{N,(1234)N,(1324)N,(1243)N,(123)N,(132)N\},
\]
\[
H/N=\{N,(1324)N\},
\]
\[
[(1234)N]_{{\theta}}=\{(1234)N,(1324)N,(1243)N\} \text{ and } [(123)N]_{{\theta}}=\{(132)N,(123)N\}
\]
Then we conclude that $\overline{\Apr}_{G/N}(H/N)=\{N,(1234)N,(1324)N,(1243)N\}$. It implies that $|\overline{\Apr}_{G/N}(H/N)|=4$ which does not divide $6$. So $\overline{\Apr}_{G/N}(H/N)$ is not a subgroup of $G/N.$ Therefore $\Apr_{G/N}(H/N)$ is not an upper rough sub-group of $G/N$. Moreover $\underline{\Apr}_{G/N}(H/N)=\{N\}$ is a proper normal subgroup of $H/N.$
\end{example}

\begin{remark}
If $H$ is a subgroup of $G$ contains a normal subgroup $N$. By Example \ref{22} it follows that $\overline{\Apr}_{G/N}(H/N)$ does not satisfy the closure law.
\end{remark}

The next Corollary shows that the intersection and product of the lower approximations are also normal subgroups.
\begin{corollary}\label{14}
Let $N$ and $M$ be two normal subgroups of $G$. Let $H_i$ be a subgroup of $G$ such that $N\subseteq H_i$ for all $i=1,2$. Suppose that $M\subseteq H_1$ then each of the following is a normal subgroup of $G/N$:
\begin{itemize}
\item [(1)] $\underline{\Apr}_{G/N}(H_1/N)\underline{\Apr}_{G/N}(H_2/N).$
\item [(2)] $\underline{\Apr}_{G/N}(H_1/N)\cap \underline{\Apr}_{G/N}(H_2/N).$
\item [(3)] $\underline{\Apr}_{G/N}((H_1\cap H_2)/N).$
\item [(4)] $\underline{\Apr}_{G/NM}(H_1/NM)$.
\item [(5)] $\underline{\Apr}_{G/N\cap M}(H_1/N\cap M)$.
\end{itemize}
\end{corollary}

\begin{proof}
Note that $\underline{\Apr}_{G/N}(H_i/N)$ is a normal subgroup of $G/N$  for all $i=1,2$ (see Proposition \ref{4}). Then it follows that
\[
\underline{\Apr}_{G/N}(H_1/N)\cap \underline{\Apr}_{G/N}(H_2/N) \text{ and } \underline{\Apr}_{G/N}(H_1/N)\underline{\Apr}_{G/N}(H_2/N)
\]
both are normal subgroups of $G/N$. Also note that $H_1\cap H_2$ is a subgroup of $G$ containing $N$. Again Proposition \ref{4} implies that $\underline{\Apr}_{G/N}((H_1\cap H_2)/N)$ is a normal subgroup of $G/N$.

Furthermore note that $NM$ and $N\cap M$ both are normal subgroups of $G$ contained in $H_1$. So it follows that $\underline{\Apr}_{G/NM}(H_1/NM)$
and $\underline{\Apr}_{G/N\cap M}(H_1/N\cap M)$ are normal subgroups of $G/N$.
\end{proof}

\section{Homomorphism between lower approximations}
Let $N$ and $M$ be any two normal subgroups of $G$. Throughout this section we will denote $\theta_1$ and $\theta$ by the conjugacy relations over $G/M$ and $G/N$ respectively. Let $H$ be any subset of $G$ containing $N$ and $M$. Here we will relate the lower and upper approximations of $H/N$ and $H/M$. Moreover we are able to develop some homomorphisms between the lower approximations of $H/N$ and $H/M$.

\begin{theorem}\label{19}
Suppose that $N\subseteq M$ are normal subgroups of $G$ and $H\subseteq G$ such that $M\subseteq H$. Let $x\in G$ be any fixed element. If $xN\in \underline{\Apr}_{G/N}(H/N)$ (resp. $xN\in \overline{\Apr}_{G/N}(H/N)$) then
\[
xM\in \underline{\Apr}_{G/M}(H/M) \text{ (resp. } xM\in \overline{\Apr}_{G/M}(H/M)).
\]
\end{theorem}

\begin{proof}
First of all let $N\subseteq M$ and $xN \in \underline{\Apr}_{G/N}(H/N)$. Let $gxg^{-1}M\in [xM]_{{\theta}_1}$ be an arbitrary element with $g\in G$. Since $gxg^{-1}N\in [xN]_{{\theta}}\subseteq H/N$. So there exists $h\in H$ such that $gxg^{-1}N=hN$. It implies that
\[
gxg^{-1}h^{-1}\in N\subseteq M.
\]
Since $M$ is a subgroup so we have $gxg^{-1}M= hM$. Then $gxg^{-1}M\in H/M$. But $gxg^{-1}M\in [xM]_{{\theta}_1}$ was arbitrary so $[xM]_{{\theta}_1}\subseteq H/M$. By definition of the lower approximation of $H/M$ we have $xM\in \underline{\Apr}_{G/M}(H/M)$.

Now let $xN\in \overline{\Apr}_{G/N}(H/N)$ then $[xN]_{{\theta}}\cap H/N\neq \varnothing.$ There exist $g,y \in G$ with $yN\in H/N$ and $xN=gyg^{-1}N$. Then $yN=hN$ where $h\in H$. It implies that $yh^{-1}\in N\subseteq M$ and $x(gyg^{-1})^{-1}\in N\subseteq M.$ Then
\[
yM=hM \text{ and } xM=gyg^{-1}M.
\]
So we have $yM\in H/M$ such that $xM\theta_1 yM$. By definition of the upper approximation we get that $[xM]_{{\theta}_1}\cap H/M\neq \varnothing$ and $xM\in \overline{\Apr}_{G/M}(H/M)$. This completes the proof of the Theorem.
\end{proof}

The next Proposition shows that the converse of Theorem \ref{19} also holds under the additional assumption of $H$ is a subgroup of $G$.
\begin{proposition}\label{20}
Let $N$ and $M$ be two normal subgroups of $G$  and $H$ a subgroup of $G$ containing $N$ and $M$. Then for any $x\in G$ the following conditions are equivalent:
\begin{itemize}
\item [(1)] $xN\in \underline{\Apr}_{G/N}(H/N)$.
\item [(2)] $xM\in \underline{\Apr}_{G/M}(H/M).$
\item [(3)] $xNM\in \underline{\Apr}_{G/NM}(H/NM).$
\item [(4)] $x(N\cap M)\in \overline{\Apr}_{G/N\cap M}(H/N\cap M).$
\end{itemize}
\end{proposition}

\begin{proof}
$(1)\Rightarrow (2)$. Suppose that $xN\in \underline{\Apr}_{G/N}(H/N)$ then it follows that $[xN]_{\theta}\subseteq H/N$. Let $gxg^{-1}M$ by any element of $ [xM]_{\theta_1}$ where $g\in G$. Note that there exists $h\in H$ such that $gxg^{-1}N=hN$. It implies that
\[
gxg^{-1}h^{-1}\in N\subseteq H.
\]
Since $H$ is a subgroup of $G$ and $h\in H$ it implies that $gxg^{-1}\in H$. So we have $gxg^{-1}M\in H/M$. This proves that $[xM]_{\theta_1}$ is a subset of $ H/M$. Therefore $xM\in \underline{\Apr}_{G/M}(H/M).$ By interchanging the role of $N$ and $M$ we can prove that $(2)$ implies $(1)$.

Note that by the above same arguments we can prove that $(1)$ is also equivalent to $(3)$ and $(4)$. Hence this finishes the proof of the Proposition.
\end{proof}

\begin{corollary}\label{17}
Let $N$ and $M$ be two normal subgroups and $H$ any subset of $G$ containing both $N$ and $M$. Let $x\in G$ be a fixed element. Then the following conditions hold:
\begin{itemize}
\item [(1)] Suppose that $NM\subseteq H$. If $xN\in \underline{\Apr}_{G/N}(H/N)$ or $xM\in \underline{\Apr}_{G/M}(H/M)$ then
\[
xNM\in \underline{\Apr}_{G/NM}(H/NM).
\]
\item [(2)] Suppose that $NM\subseteq H$. If $xN\in \overline{\Apr}_{G/N}(H/N)$ or $xM\in \overline{\Apr}_{G/M}(H/M)$ then
\[
xNM\in \overline{\Apr}_{G/NM}(H/NM).
\]
\item [(3)] If $x(N\cap M)\in \underline{\Apr}_{G/N\cap M}(H/N\cap M)$ then
\[
xN\in \underline{\Apr}_{G/N}(H/N) \text{ and } xM\in \underline{\Apr}_{G/M}(H/M).
\]
\item [(4)] If $x(N\cap M)\in \overline{\Apr}_{G/N\cap M}(H/N\cap M)$ then
\[
xN\in \overline{\Apr}_{G/N}(H/N)\text{ and } xM\in \overline{\Apr}_{G/M}(H/M).
\]
\end{itemize}
\end{corollary}

\begin{proof}
It is straightforward in view of the Theorem \ref{19}. Note that $NM$ and $N\cap M$ both are normal subgroups of $G$.
\end{proof}

Before proving the next result we need some preparation. Let $N$ and $ M$ be two normal subgroups of $G$ and $H$ a subgroup of $G$ containing $N$ and $M$. Since $\underline{\Apr}_{G/N}(H/N)$ and $\underline{\Apr}_{G/M}(H/M)$ are normal subgroups of $G/N$ and $G/M$ respectively (see Proposition \ref{4}). So by one-one correspondence Theorem $\underline{\Apr}_{G/N}(H/N)$ and $\underline{\Apr}_{G/M}(H/M)$ are of the form $K/N$ and $T/M$ respectively. Here $K$ and $T$ are normal subgroups of $G$ such that $N\subseteq K\subseteq H$ and $M\subseteq T\subseteq H$.

\begin{theorem}\label{8}
With the above notation suppose that $N\subseteq M$. Then the following results are true:
\begin{itemize}
\item [(1)] There is a group isomorphism
\[
\frac{G/N}{K/N}\to \frac{{G/M}}{{T/M}}, xN(K/N)\mapsto xM(T/M).
\]
In particular $G/K$ is isomorphic to $G/T$.
\item [(2)] There is a group isomorphism
\[
\frac{H/N}{K/N}\to \frac{{H/M}}{{T/M}}, xN(K/N)\mapsto xM(T/M).
\]
In particular $H/K$ is isomorphic to $H/T$.
\item [(3)] There is an onto group homomorphism
\[
K/N\to T/M, xN\mapsto xM.
\]
with kernel is equal to $M/N.$ In particular we have the following inclusions:
\[
M/N\subseteq K/N\subseteq H/N\text{ and } M\subseteq K\subseteq H.
\]
\item [(4)] $K/M$ is isomorphic to $T/M$.
\end{itemize}
\end{theorem}

\begin{proof}
$(1)$ Since $N\subseteq M$ so it induces an onto group homomorphism $f: G/N\to G/M,xN\mapsto xM$. We claim that $f(K/N)\subseteq T/M$. If $xM\in f(K/N)$ then there exists $yN\in K/N$ such that $yM=f(yN)=xM$. Since $H$ is a subgroup of $G$ and $yN\in K/N$ so it follows that $yM\in T/M$ (see Proposition \ref{20}). This proves that $xM\in T/M$ and $f(K/N)\subseteq T/M$. So we have proved the claim. Then it is well known that $f$ induces the following group homomorphism
\[
\psi: \frac{G/N}{K/N}\to \frac{{G/M}}{{T/M}},
\]
with $\psi(xN(K/N))= f(xN)(T/M)=xM(T/M)$. Since $f$ is onto it follows that $\psi$ is onto. We only need to prove that $\psi$ is injective. Let $xN\in \ker(\psi)$ then
\[
\psi(xN(K/N))=xM(T/M)=T/M.
\]
It implies that $xM\in T/M$. By Proposition \ref{20} we get that $xN\in K/N$. Recall that $H$ is a subgroup of $G$. This proves that $\psi$ is an isomorphism. Moreover by second isomorphism Theorem we get that $G/K$ is isomorphic to $G/T$. By following the same steps we can prove the isomorphisms in $(2)$.

$(3)$ Let $\varphi: K/N\to T/M$ be defined as $\varphi(xN)= xM$ for all $xN\in K/N$. Firstly we will show that $\varphi$ is well-defined. For this let $x_1N=x_2N$ with $x_iN\in K/N$ for all $i=1,2.$ By Theorem \ref{19} it implies that $x_1x_2^{-1}\in N\subseteq M$ such that $x_iM\in T/M$ for all $i=1,2.$ So we conclude that
\[
\varphi(x_1N)= x_1M=x_2M=\varphi(x_2N).
\]
This proves that $\varphi$ is well-defined. Moreover
\[
\varphi((x_1N)(x_2N))=\varphi(x_1x_2N)= x_1x_2M=(x_1M)(x_2M)=\varphi(x_1N)\varphi(x_2N)
\]
since $N$ and $M$ are normal. So $\varphi$ is a group homomorphism. By Proposition \ref{20} it is easy to see that $\varphi$ is onto. Now we calculate the $\ker(\varphi)$ as follows:
\[
\ker(\varphi)=\{xN:\varphi(xN)=M\}=\{xN:xM=M\}=\{xN:x\in M\}=M/N.
\]
This proves the claim in $(2)$ in view of one-one correspondence Theorem. Moreover $(4)$ follows from $(3)$ in view of the second isomorphism Theorem.
\end{proof}

\begin{corollary}\label{}
Let $N$ and $M$ be two normal subgroups of $G$ contained in a subgroup $H$. Then we have:
\begin{itemize}
\item [(1)] There is an onto group homomorphism
\[
\underline{\Apr}_{G/(N\cap M)}(H/(N\cap M))\to \underline{\Apr}_{G/N}(H/N), x(N\cap M)\mapsto xN.
\]
with kernel is equal to $N/(N\cap M).$
\item [(2)] There is an onto group homomorphism
\[
\underline{\Apr}_{G/N}(H/N)\to \underline{\Apr}_{G/NM}(H/NM), xN\mapsto x(NM).
\]
with kernel is equal to $NM/N.$
\item [(3)] There is an onto group homomorphism
\[
\underline{\Apr}_{G/(N\cap M)}(H/(N\cap M))\to \underline{\Apr}_{G/NM}(H/NM), x(N\cap M)\mapsto xNM.
\]
with kernel is equal to $NM/(N\cap M).$
\end{itemize}
\end{corollary}

\begin{proof}
Since $(N\cap M)$, $NM$ both are normal subgroups of $G$ contained in $H$ and $(N\cap M)\subseteq N\subseteq NM$. Then all the claims of the Corollary are obvious in view of Theorem \ref{8}$(2)$.
\end{proof}



\begin{thebibliography}{9999}
\bibitem{b} {\sc R. Biswas and S. Nanda:} Rough groups and rough subgroups, Bull. Polish Acad. Sci. Math 42 (1994), 251-254.

\bibitem{b2} {\sc Z. Bonikowaski:} Algebraic structures of rough sets. In W. P. Ziarko (Ed.), Rough sets, fuzzy sets and knowledge discovery, (1995), (pp. 242-247). Berlin: Springer-Verlag.

\bibitem{b1} {\sc XU Bi-cai :} Anti-homorphism on a group and its applications, Third international conference on advancd computer theory
and Eenginering, 2010.

\bibitem{b3} {\sc M. K. Chakraborty, M. Banergee :} Logic and algebra of the rough sets. In W. P. Ziarko (Ed.), Rough sets, fuzzy sets and knowledge discovery, (1994), (pp. 196-207). London: Springer-Verlag.

\bibitem{d} {\sc Wang De-song:} Application of the theory of rough set on the groups and rings, Masters thesis, 2004, (dissertation for Master
Degree).

\bibitem{da} {\sc B. Davvaz:} Roughness in rings, Inform. Sci. 164 (2004), 147-163.

\bibitem{da2} {\sc B. Davvaz:} Roughness based on fuzzy ideals, Inform. Sci. 164 (2006) 2417-2437.

\bibitem{da1} {\sc B. Davvaz, M. Mahdavipour:} Roughness in modules. Information Sciences, 176 (2006), 3658-3674.

\bibitem{g} {\sc J. K. Goldhaber. G. Enrich:} The homomorphism and isomorphism of rough groups. Academy of Shanxi University, 24 (2001), 303305.

\bibitem{h} {\sc S. Han, W. Cheng, J. Wang:} Rough rings revisited. Proccedings of the fourth international conference on machine learning and cybernetics, guangzhou., 18-21, August 2005, pp. 3157-3161.

\bibitem{i} {\sc P. Isaac, Neelima:} Rough ideals and their properties, J. of global research in Mathematical archives, Vol. 6 (2013),

\bibitem{k} {\sc N. Kuroki, P.P. Wang:} The lower and upper approximations in a fuzzy group, Inform. Sci. 90 (1996) 203-220.

\bibitem{k1} {\sc N. Kuroki:} Rough ideals in semigroups, Inform. Sci. 100 (1997) 139-163.

\bibitem{k2} {\sc N. Kuroki, J.N. Mordeson:} Structure of rough sets and rough groups, J. Fuzzy Math. 5 (1)(1997) 183-191.

\bibitem{m} {\sc D. Miao, S. Han, Daoguo Li, L. Sun:} Rough group, rough subgroup and their proprties, Proccedings of
RSFDGrC. (D. $\acute{S}$lezak et al., ed.), Springer-Verlag Berlin Heidelberg, 2005, pp. 104-113.

\bibitem{p} {\sc  Z. Pawlak:} Rough sets. Int. Jl. Inform. Comput. Sci. 11 (1982), 341-356.

\bibitem{p1} {\sc Z. Pawlak:} Rough Sets  Theoretical Aspects of Reasoning about Data, Kluwer Academic Publishing, Dordrecht, 1991.

\bibitem{w} {\sc C. Wanga, D. Chen:} A short note on some properties of rough groups. Computers and Mathematics with Applications 59 (2010) 431-436.

\bibitem{q} {\sc Q.M. Xiao, Z.L. Zhang:} Rough prime ideals and rough fuzzy prime ideals in semigroups. Inform. Sci. 176 (2006) 725-733.
\end{thebibliography}
\end{document}